\documentclass[reqno, letterpaper]{amsart}

\usepackage[foot]{amsaddr}
\usepackage[margin=1.6in,bottom=1.5in]{geometry} 
\usepackage{titlesec} 
\titleformat{\section}            
  {\centering\normalfont\scshape\large} 
  {\thesection}{1em}{}            
\titleformat{\subsection}         
  {\normalfont\bfseries}          
  {\thesubsection}{1em}{}         

\usepackage{amsfonts,amssymb,amsmath,amscd,amstext,bbm, enumerate, dsfont, mathtools} 

\usepackage{hyperref} 
\hypersetup{colorlinks=true, urlcolor=blue, linkcolor=blue, citecolor=blue}
\usepackage{cleveref} 

\usepackage{graphicx} 
\usepackage{tikz,pgfplots} 
\usepgfplotslibrary{groupplots, patchplots, fillbetween} 
\pgfplotsset{compat=1.15} 

\usepackage{array} 
\usepackage{multicol, multirow} 

\usepackage[algo2e]{algorithm2e} 
\usepackage{algorithm} 

\usepackage{tcolorbox} 
\usepackage{parskip} 
\usepackage{svg} 

\makeatletter
\let\cl@chapter\undefined
\makeatother

\pgfplotsset{%
    layers/standard/.define layer set={%
        background,axis background,axis grid,axis ticks,axis lines,axis tick labels,pre main,main,axis descriptions,axis foreground%
    }{
        grid style={/pgfplots/on layer=axis grid},%
        tick style={/pgfplots/on layer=axis ticks},%
        axis line style={/pgfplots/on layer=axis lines},%
        label style={/pgfplots/on layer=axis descriptions},%
        legend style={/pgfplots/on layer=axis descriptions},%
        title style={/pgfplots/on layer=axis descriptions},%
        colorbar style={/pgfplots/on layer=axis descriptions},%
        ticklabel style={/pgfplots/on layer=axis tick labels},%
        axis background@ style={/pgfplots/on layer=axis background},%
        3d box foreground style={/pgfplots/on layer=axis foreground},%
    },
}

\newtheorem{theorem}{Theorem}
\crefname{theorem}{Theorem}{Theorems}

\newtheorem{lemma}{Lemma}
\crefname{lemma}{Lemma}{Lemmas}

\newtheorem{asm}{Assumption}
\crefname{asm}{Assumption}{Assumptions}

\newtheorem{example}{Example}
\crefname{example}{Example}{Examples}

\DeclareMathOperator*{\argmin}{argmin}
\DeclareMathOperator*{\dom}{dom}
\DeclareMathOperator*{\interior}{int}
\newcommand{\Constraints}{\mathcal{X}}
\newcommand{\ProdSpace}{\mathbb{X}}

\begin{document} 

\title{Delay-tolerant Distributed Bregman Proximal Algorithms}

\author[S. Chraibi]{S. Chraibi}
\address[S. Chraibi]{Univ. Grenoble Alpes, LJK, 38000 Grenoble, France}
\email[S. Chraibi]{selim.chraibi@univ-grenoble-alpes.fr}

\author[F. Iutzeler]{F. Iutzeler}
\address[F. Iutzeler]{Univ. Grenoble Alpes, CNRS, Grenoble INP, LJK, 38000 Grenoble, France}

\author[J. Malick]{J. Malick}
\address[J. Malick]{Univ. Grenoble Alpes, CNRS, Grenoble INP, LJK, 38000 Grenoble, France}

\author[A. Rogozin]{A. Rogozin}
\address[A. Rogozin]{Moscow Institute of Physics and Technology, Moscow, Russia}

\maketitle

\begin{abstract}
    Many problems in machine learning write as the minimization of a sum of individual loss functions over the training examples. These functions are usually differentiable but, in some cases, their gradients are not Lipschitz continuous, which compromises the use of (proximal) gradient algorithms. Fortunately, changing the geometry and using Bregman divergences can alleviate this issue in several applications, such as for Poisson linear inverse problems.
    However, the Bregman operation makes the aggregation of several points and gradients more involved, hindering the distribution of computations for such problems.
    In this paper, we propose an asynchronous variant of the Bregman proximal-gradient method, able to adapt to any centralized computing system.
    In particular, we prove that the algorithm copes with arbitrarily long delays and we illustrate its behavior on distributed Poisson inverse problems.

\end{abstract}

\section{Introduction}\label{sec:intro}

\subsection{Context: distributed problems and lack of smoothness}
Many problems in machine learning and signal processing involve the minimization of a sum of functions measuring the loss between the model and the data points.
This sum form is highly practical when the data is distributed over several machines: the total loss is simply the sum of the losses over each machine's local data, and, similarly, the gradient of the total loss is the sum of the machines' local gradients. Thus, a central machine can minimize the total loss using a first-order method by simply querying the machines for their local gradients.
Such a method is called \emph{synchronous} since, at each iteration, the central machine waits for all the machines to respond before proceeding to a new query.
Unfortunately, the time a machine takes to respond can vary a lot due to network/communication issues, uneven data, or heterogeneous computing power.
This calls for \emph{asynchronous} methods where the central machine updates the global model and queries a new gradient, as soon as a machine responds.
Despite their more involved analysis, asynchronous optimization methods have become popular thanks to their practical efficiency. Indeed, having asynchronous exchanges is crucial in practice since then much more iterations can be performed in the same time interval compared to synchronous setups,
see eg. the discussion in\;\cite{hannah2017more}.

Most existing analyses of asynchronous first-order methods assume that the functions are \emph{smooth} (ie. that their gradients are Lipschitz continuous). While such an assumption often holds, several objectives of interest do not satisfy it, even though they are differentiable; it is the case, for example, for recovery from quadratic measurements \cite{bolte2018first},  and Poisson inverse problems \cite{bertero2009image}.
This lack of smoothness breaks down the usual ``descent lemmas'' at the core of the analysis of first-order methods; see eg. \cite[Chap. 10]{beck2017first}. Fortunately, for several problems of interest, including the two mentioned above, changing the geometry can alleviate the issue of lack of smoothness. The idea is to use, instead of Euclidean geometry, the geometry induced by so-called \emph{Bregman divergences} \cite{bregman1967relaxation,bauschke1997legendre}. Indeed, a descent/contraction lemma can be obtained for functions that are smooth with respect to a Bregman divergence \cite{bauschke2017descent}, such as in the two examples above.
This opens the way for the application and analysis of first-order methods, as we recall in \Cref{sec:smooth}.  Our work in this paper can be included in this line of research: we develop and analyze a \emph{distributed version} of the Bregman proximal  method of \cite{bauschke2017descent}.

\subsection{Asynchronous Bregman proximal-gradient in centralized set-up}
We consider a centralized setup where a central machine coordinates asynchronously $M$ worker machines to solve an optimization problem of the form
\begin{align}
    \label{eq:pb}
    \tag{$\mathcal{P}$}
    \min_{x\in\Constraints}~~ F(x) := \frac{1}{M}\sum_{i=1}^M f_i(x) + g(x).
\end{align}

We make no special assumption on the underlying system which can be completely heterogeneous.
In the optimization problem, each function $f_i$ is local to machine~$i$, and the central machine shares the search space $\Constraints\subset\mathbb{R}^n$ and a regularization function~$g$.
We focus on the case where $g$ is convex and lower semi-continuous (\emph{lsc}), and that $f_i : \Constraints\to\mathbb{R}$ is convex and differentiable -- but not necessarily smooth (ie. we do \emph{not} assume that the gradients are Lipschitz continuous on $\Constraints$).

The lack of smoothness ruins the theoretical properties of existing asynchronous distributed proximal-gradient algorithms (see e.g.,\;\mbox{\cite{mishchenko2020distributed, vanli2018global}} and references therein) in the same way that it ruins the convergence of standard (non-distributed) proximal gradient methods (which has led to the works of \cite{bauschke2017descent} and \cite{lu2017relativelysmooth}). In particular, \cite{bauschke2017descent} proposes a Bregman proximal-gradient method that fits our assumptions. However, an asynchronous distributed version of this method still has to be designed and analyzed, and this is what we propose in this paper.

More specifically, we use the tools developed in \cite{bauschke2017descent} to extend the asynchronous proximal-gradient algorithm of \cite{mishchenko2020distributed} when the local functions $f_i$ are not smooth in the Euclidean geometry, but rather in an adapted Bregman geometry.
The resulting Bregman proximal-gradient algorithm copes with flexible asynchronous communications, and we analyze its convergence under mild assumptions on both the functions and the computing setup. First, we rely on the same set of assumptions as in \cite{bauschke2017descent} on the functions which enable us to control the \emph{Bregman divergence} between iterates and the solution. This change of metric (and in particular the loss of symmetry) introduces technical challenges compared to  \cite{mishchenko2020distributed, vanli2018global} which rely on controlling the \emph{Euclidean distance}.
Second, we pay special attention to having realistic assumptions on communication delays to cover a diversity of the scenarios included in this framework (e.g. computing clusters, mobile devices):
we do \emph{not} assume that the delays between communications are bounded and we only assume that all workers eventually communicate with the central machine, as per the definition of \emph{totally asynchronous} methods by Bertsekas and Tsitsiklis' classification \cite[Chap.\;6.1]{bertsekas1989parallel}. To the best of our knowledge, most papers on asynchronous distributed methods rely on a bounded delay assumption; exceptions include \cite{sun2017asynchronous,hannah2017more,hannah2016unbounded,icml2018,mishchenko2020distributed} but they rely on the smoothness of the objective that we do not have here.

The outline of the paper is as follows. In \Cref{sec:recalls}, we introduce the notation and recall the ideas of, both, Bregman smoothness \cite{bauschke2017descent} and standard asynchronous proximal-gradient algorithms \cite{vanli2018global, mishchenko2020distributed}.
In \Cref{sec:algo}, we present our asynchronous Bregman algorithm, by adapting the developments of the Euclidean setting. Then we analyze in \Cref{sec:analysis} the well-posedness and the convergence of our algorithm. Notably, we prove convergence for a fixed stepsize, independent of the computing system (in particular independent of delays in the system). Finally, we provide, in \Cref{sec:num}, numerical illustrations of the behavior of the algorithm on distributed Poisson inverse
problems.

\section{Notation and recalls on proximal-gradient and Bregman geometry}\label{sec:recalls}

This section introduces the notation used in this paper and recalls the main notions and ideas.
\Cref{sec:from} presents the natural splitting of the proximal-gradient method that serves as a basis for the developments of the next section. It also introduces important notation about delays that will be constantly used in the sequel. \Cref{sec:smooth} presents how Bregman smoothness can replace the usual smoothness as a fundamental tool for convergence analysis.

\subsection{From proximal-gradient to asynchronous optimization}\label{sec:from}

A natural rationale to distribute algorithm is to first endow each worker with a copy $x_i$ of the global variable and impose a consensus through the indicator function $\iota_C : \Constraints^M \to \overline{\mathbb{R}}$ defined as
\begin{equation}\label{eq:C}
    \text{$\iota_C(x_1,\dots,x_M) = 0 $ if $x_i=x_j$ for all $i,j$ and $+\infty$ elsewhere.}
\end{equation}
Mathematically, we end up with the following problem equivalent \footnote{Equivalent here means that $x^\star$ is a solution of \eqref{eq:pb} if and only if $(x^\star,\dots,x^\star)$ is a solution of \eqref{eq:pb_dist}.}
to \eqref{eq:base}
\begin{align}
    \label{eq:pb_dist}
    \min_{(x_1,\dots,x_M)\in\Constraints^M}~~ \frac{1}{M}\sum_{i=1}^M f_i(x_i) + g(x_i) + \iota_C(x_1,\dots,x_M)
\end{align}
in which the first term is differentiable while the two others are convex and lower semi-continuous.

Such a problem naturally calls for proximal gradient methods (where the differentiable part of the objective is iteratively replaced by a quadratic model). Specifically, for a given step size $\gamma$
\begin{multline}
    (x_1^k,\dots,x_M^k) = \argmin_{(x_1,\dots,x_M)\in\Constraints^M}
    \left\{ \frac{1}{M}\sum_{i=1}^M \left( f_i(x^{k-1}_i) + \langle x_i - x_i^{k-1} ; \nabla f_i(x_i^{k-1}) \rangle + \dots \vphantom{\left\|\right\|^2} \right. \right. \\
    \left. \left. \dots + \frac{1}{2\gamma} \left\| x_i - x_i^{k-1} \right\|^2 + g(x_i) \right)
    + \iota_C(x_1,\dots,x_M) \vphantom{\sum_{i=1}^M} \right\}
\end{multline}
where the particular form of $\iota_C$ immediately leads to $x_1^k=\dots=x_M^k$. Thus we have, for all $i\in 1,\dots,M$,
\begin{align}
    x_i^k
     & = \argmin_{x\in\Constraints} \left\{ g(x) +  \frac{1}{M}\sum_{i=1}^M \left( f_i(x^{k-1}_i) + \langle x - x_i^{k-1} ; \nabla f_i(x_i^{k-1}) \rangle + \frac{1}{2\gamma} \left\| x - x_i^{k-1}  \right\|^2 \right)  \right\} \\
     & = \argmin_{x\in\Constraints} \left\{ g(x) +  \left\| x - \frac{1}{M}\sum_{i=1}^M  \left(x_i^{k-1} - \gamma \nabla f_i(x_i^{k-1}) \right)  \right\|^2 \right\}  .\label{eq:gradprox}
\end{align}

In terms of distributed optimization, this means that the central machine has to gather the gradient steps of all the workers  ($x_i^{k-1} - \gamma \nabla f_i(x_i^{k-1})$), average them, and perform a proximal operation for $g$ on $\Constraints$ to get the next iterate, then send the result to all workers.
This algorithm is by construction completely synchronous.

Let us now derive an asynchronous version of this algorithm. To do so, we need some notation, introduced below and illustrated in \Cref{fig:scheme}.
First, we call iteration (or time) $k$ the moment of the $k$-th exchange between a worker and the central machine. Second, we denote by $d_i^k$ the delay suffered by worker $i$ at time $k$, defined by the number of iterations since worker $i$ last exchanged with the central machine. We also define the \emph{second-order} delay for worker~$i$ and time~$k$ by
\begin{equation}\label{eq:D}
    D_i^k = d_i^k + d_i^{k-d_i^k-1}+1.
\end{equation}
These two delays allow us to handle, at each time $k$,
the
worker\;$i$'s gradient previously received by the central machine at the last exchange $k-d_i^k$, which was itself computed at a point of the \emph{second} last exchange $x^{k-D_i^k}$.

\begin{figure}[!h]
    \centering
    \resizebox{0.9\textwidth}{!}{\includegraphics{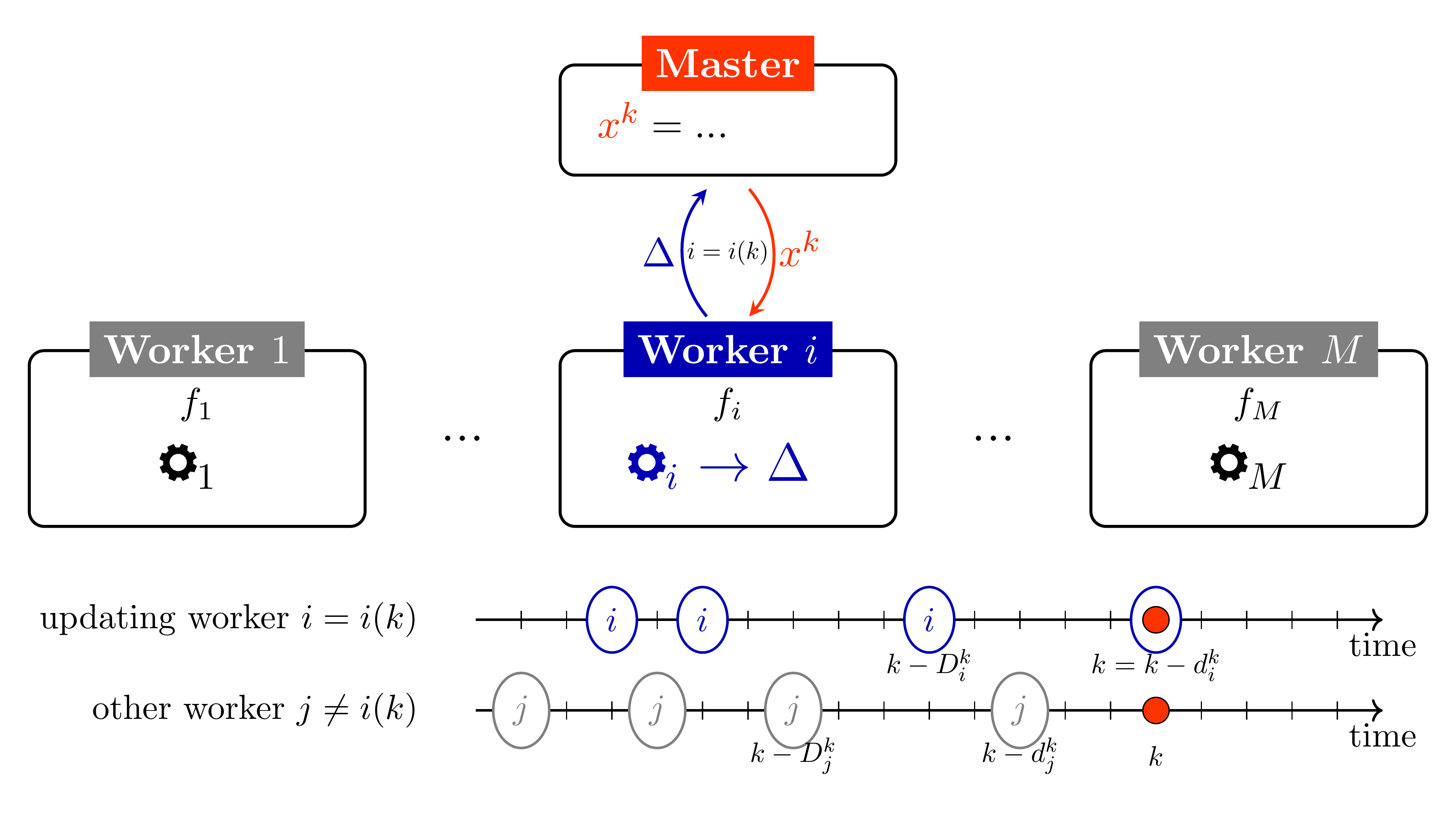}}
    \vspace*{-2ex}
    \caption{Asynchronous distributed setting and notation. Top: As soon as a worker finishes its computation, it sends its updates to the master, gets a new query point, and proceeds to the next computation. Bottom about delays notation: $d_i^k$ is the delay suffered by worker $i$ at time $k$ defined as the number of iterations since worker $i$ last exchanged with the master (the exchanging worker $i=i(k)$ has no delay); the second-order delay is $D_i^k$ corresponds to the second-last exchange.}
    \label{fig:scheme}
\end{figure}

To get an asynchronous variant of \eqref{eq:gradprox}, the iteration has to be modified as follows.
Regarding the use of gradient, an asynchronous version has to replace $ \nabla f_i(x^{k-1})$ by $ \nabla f_i(x^{k-D_i^k})$ in the iteration. Regarding the base point, there are two possibilities:
\begin{itemize}
    \item either it is kept at the last iterate $x^{k-1}$, which leads to the update of the proximal incremental aggregated gradient (PIAG) method \cite{aytekin2016analysis,vanli2018global}:
          \begin{align}
              x^k & = \argmin_{x\in\Constraints} \left\{ g(x) +  \left\| x - \frac{1}{M}\sum_{i=1}^M  \left(x^{k-1} - \gamma \nabla f_i(x^{k-D_i^k}) \right)  \right\|^2 \right\} ;
          \end{align}
    \item or it is set to the point corresponding to the gradient computation $x^{k-D_i^k}$, which leads to the update of DAve-PG \cite{icml2018,mishchenko2020distributed}:
          \begin{align}
              x^k & = \argmin_{x\in\Constraints} \left\{ g(x) +  \left\| x - \frac{1}{M}\sum_{i=1}^M  \left(x^{k-D_i^k} - \gamma \nabla f_i(x^{k-D_i^k}) \right)  \right\|^2 \right\} .
          \end{align}
\end{itemize}
For our analysis, we focus here on the second option which is often faster in the Euclidean case, as shown in \cite[Sections 2.4 and 3.4]{mishchenko2020distributed}). We will come back to this distinction in the numerical illustrations.

\subsection{Descent with and without smoothness}\label{sec:smooth}

Smoothness is a central property to obtain functional descent and contraction in first-order methods \cite{nesterov-book,beck2017first}. Let a function $f:\mathbb{R}^n\to \mathbb{R}$ be convex and differentiable. If $f$ is furthermore $L$-smooth, which means
\begin{equation}\label{eq:smooth}
    \text{$\nabla f$ is $L$-Lipschitz continuous or, equivalently,
        $L\|\cdot\|_2^2/2-f$ is convex},
\end{equation}
then we have that (see eg. \cite[Lem.~5.]{beck2017first})
\begin{align}
    f(y) \leq f(x) + \langle y-x ; \nabla f(x) \rangle + \frac{L}{2} \| y - x\|_2^2 ~~\text{ for all } x,y\in\mathbb{R}^n
\end{align}
Thus the gradient step, defined as
\begin{align}
    p_\gamma(x) & = \argmin_{y} \left\{  f(x) + \langle y-x ; \nabla f(x) \rangle + \frac{1}{2\gamma} \| y - x\|_2^2 \right\}, \label{eq:grad} \\
                & = x - \gamma \nabla f(x),
\end{align}
grants both descent and contraction (see eg. \cite[Th. 18.15]{bauschke2011convex}): for all $x\in\mathbb{R}^n$ and for a minimizer $x^\star$ of $f$, we have
\begin{align}
    f(p_\gamma(x) )                              & \leq f(x)   - \gamma \left(1- \frac{L\gamma}{2}\right) \| \nabla f(x) \|_2^2, \quad\text{ and }                                                       \\
    \left\|  p_\gamma(x) - x^\star  \right\|_2^2 & \leq    \left\|  x - x^\star  \right\|_2^2 -   \gamma \left(\frac{2}{L} - \gamma \right)   \left\| \nabla f(x) \right\|_2^2 \label{eq:contract:eucl}.
\end{align}

When smoothness is not present, these inequalities do not hold anymore, which hinders the analysis of gradient methods. Fortunately, for a function $f\colon\Constraints\to\mathbb{R}$, convex and differentiable on the interior of $\Constraints$, some smoothness can be obtained by comparing it to a Bregman function; see eg.~\cite[Ch.\;26]{rockafellar2015convex}.
\begin{asm}[Bregman regularizer]\label{asm:reg}
    The function $h\colon\mathbb{R}^n\to\mathbb{R}$ satisfies the following conditions:
    $h$ is proper, lower semi-continuous, and convex;
    ${\dom h} = \Constraints$;
    $h$ is of Legendre type\footnote{We recall that a Legendre function is differentiable and strictly convex on $\interior\dom h$ with $\|\nabla h(x^k)\| \to +\infty$ for $(x^k)$ of  $\interior\dom h$ converging to a boundary point of $\dom h$.}.
\end{asm}

Under this assumption, we have that $\dom \partial h = \interior\dom h$ and that for all $x\in \interior\dom h$, $\partial h(x) = \{\nabla h(x)\}$.
Building on this function $h$, we can define the Bregman distance\;\cite{bregman1967relaxation}
\begin{align}
    D_h(x,y) = h(x) - h(y) -\langle \nabla h(y) , x-y \rangle
\end{align}
for any $x\in\dom h$ and any $y\in\interior\dom h$.
We also define Bregman gradient steps \cite{auslender2006interior}
for $x\in\interior \dom h$, similarly as \eqref{eq:grad}, by
\begin{align}
    p_\gamma(x) & = \argmin_{y\in\dom h} \left\{  f(x) + \langle y-x ; \nabla f(x) \rangle + \frac{1}{\gamma} D_h(y,x) \right\} \\
                & = \argmin_{y\in\dom h} \left\{  h(y)  + \langle y ; \gamma \nabla f(x) - \nabla h(x) \rangle \right\}.
\end{align}
The simple yet powerful idea of \cite{bauschke2017descent} is then to compare $f$ to $h$ in order to extend smoothness \eqref{eq:smooth} beyond the Euclidean case: if there is an $L>0$ such that $Lh-f$ is convex on $\interior\dom h$, then \cite[Lem. 1]{bauschke2017descent}
\begin{align}
    f(y) \leq f(x) + \langle y-x ; \nabla f(x) \rangle + L D_h(y,x) ~~\text{ for all } x,y\in\interior\dom h .
\end{align}
Therefore, counterparts of \eqref{eq:contract:eucl} hold
\begin{align}
    f(p_\gamma(x) )          & \leq f(x)   - \frac{1}{\gamma} D_h(x,p_\gamma(x)) - (1-\gamma L) D_h(p_\gamma(x),x), \text{ and } \\
    D_h(x^\star,p_\gamma(x)) & \leq D_h(x^\star,x)  - (1-\gamma L) D_h(p_\gamma(x),x)\label{eq:contractionB}
\end{align}
for all $x\in\interior \dom h$ and $x^\star$ a solution of $\min_\Constraints f$,\footnote{Here, we use explicitly that $\dom h = \Constraints$.} see eg. \cite[Lem. 5]{bauschke2017descent}.

\begin{example}[Nonnegative linear regression]\label{example}
    Let $A\in\mathbb{R}_+^{m\times n}$ be matrix with non-null rows $(a_i)_{i\in 1,\ldots,m}$, $b\in\mathbb{R}_{++}^m$ a positive output vector. Nonnegative linear regression minimizes the generalized Kullback-Leibler divergence between a linear model $Ax$, with $x\in\Constraints=\mathbb{R}_+^{n}$, and the output $b$ (see eg. \cite[Sec 5.3]{bauschke2017descent}).
    \begin{align}
        \min_{x\in \mathbb{R}_+^{n}} f(x) := \text{KL}(Ax, b)
    \end{align}
    Where for $v,u \in \mathbb{R}^m_+$, $\text{KL}(v, u) = \sum_{i=1}^m v_i \log v_i / u_i  -v_i  + u_i$ and $\text{KL}(0, u) = \|u\|_1$, by continuous extension. We see that $f$ is differentiable on $\interior \Constraints$, but it is not smooth. It is however smooth relative to the Boltzmann-Shannon entropy ($h(x) = \sum_{j=1}^n x_j \log x_j$ with $\dom h = \Constraints$), since by \cite[Lem. 8]{bauschke2017descent}, $Lh-f$ is convex on $\interior\dom h$ with
    \begin{equation}\label{eq:L}
        L = \max_{j}\sum_{i=1}^m a_{ij}\ .
    \end{equation}
    Since $\nabla f(x) = \sum_{i=1}^m a_i \log( \langle a_i , x \rangle/b_i )$ and $\nabla h(x) = 1 + \log(x)$, the Bregman gradient operator $ p_\gamma(x) = \argmin_{y\in\dom h} \{  h(y)  + \langle y ; \gamma \nabla f(x) - \nabla h(x) \rangle\} $ can be computed coordinate-wise and we obtain that for any $x\in\interior\dom h$,
    \begin{align}
        \left[p_\gamma(x)\right]_j & = \argmin_{y\geq 0} \left\{  y \log y + y \left( \gamma \sum_{i=1}^m a_{ij} \log\left( \frac{\langle a_i , x \rangle}{b_i}  \right) - 1 - \log(x_j)  \right) \right\} \\
                                   & = x_j \exp \left( - \gamma \sum_{i=1}^m a_{ij} \log\left( \frac{\langle a_i , x \rangle}{b_i}  \right)    \right)
        = \frac{x_j}{\prod_{i=1}^m \left( \frac{\langle a_i , x \rangle}{b_i}  \right)^{\gamma\; a_{ij}}} .
    \end{align}
    Hence, $p_\gamma(x) \in\interior\dom h$ and brings functional descent and contraction from \eqref{eq:contractionB} with $\gamma\leq 1/L$ given by \eqref{eq:L}.
    Note that the objective admits a unique minimizer on $\Constraints$ when $A$ is full-rank. In this case indeed, observe first that $f$ is coercive which proves the existence of a minimizer on $\Constraints$.
    Second, considering two minimizers $x_1, x_2$ of $f$ on $ \mathbb{R}_+^{n}$, the strict convexity of  $
        v \mapsto \text{KL}(v, u)$ implies that $x_1-x_2\in\mathrm{ker} A  = \{0\}$, which shows the uniqueness.
    \hfill $\blacktriangleleft$

\end{example}

\section{Asynchronous Bregman proximal gradient}\label{sec:algo}

\subsection{Assumptions}

In the framework introduced in the previous section, we present our algorithm for solving \eqref{eq:base}, under the following assumptions.
First, we need to assume that the problem has minimizers, which can be done by constraining $\Constraints$ to be closed and convex and $g$ to be lower semi-continuous, and proper.

\begin{asm}[On the problem]\label{asm:pb}
    The following conditions hold:
    \begin{itemize}
        \item[i.] the constraint set $\Constraints\subseteq \mathbb{R}^n$ is closed\footnote{The closedness, combined with \cref{asm:reg}, implies that $\dom h$ is closed, which is not always assumed in the case of optimization with Bregman divergences. In our case, it will be helpful to characterize pointwise convergence (as in eg. \cite[Th. 2 ii]{bauschke2017descent}) since functional decrease is out of reach in our case (see \cite{mishchenko2020distributed}).} and convex with a non-empty interior;
        \item[ii.] the function $g$ is proper, lsc, convex, with $\dom g \cap \interior \Constraints\neq\emptyset$ ;
        \item[iii.] the minimizers of \eqref{eq:pb} form a non-empty compact set in $\Constraints$.
    \end{itemize}
\end{asm}

We also formalize our assumption that the local functions $f_i$ of the problem are smooth with respect to the Bregman divergence generated by $h$ in the sense of \cite{bauschke2017descent}.

\begin{asm}[On the functions $(f_i)$]\label{asm:fun}
    For every $i=1,\dots,m$, we assume that  the $\Constraints\to\mathbb{R}$ function $f_i$ verifies:
    \begin{itemize}
        \item[i.] $f_i$ is proper, lower semi-continuous, and convex;
        \item[ii.] $f_i$ is differentiable on $\interior\dom h$;
        \item[iii.] $f_i$ is  $L$-smooth with respect to $h$, ie. $Lh-f_i$ is convex on $\interior\dom h$.
    \end{itemize}
\end{asm}

As we will see in \Cref{sec:well}, this set of assumptions ensures that the Bregman steps at the core of our development are well-defined.

\subsection{From Bregman proximal gradient to asynchronous optimization}\label{sec:bregmanasync}

We now take another look at problem \eqref{eq:pb_dist} by considering a Bregman proximal gradient method. We can perform steps of the form
\begin{multline}
    (x_1^k,\dots,x_M^k) = \argmin_{(x_1,\dots,x_M)\in\Constraints^M}
    \left\{ \frac{1}{M}\sum_{i=1}^M \left( \vphantom{\frac{1}{\gamma}} f_i(x^{k-1}_i) + \langle x_i - x_i^{k-1} ; \nabla              f_i(x_i^{k-1}) \rangle + \dots \right.\right. \\
    \left.\left. \dots + \frac{1}{\gamma} D_h(x_i,x_i^{k-1})  + g(x_i) \right)
    + \iota_C(x_1,\dots,x_M)  \vphantom{\frac{1}{M}\sum_{i=1}^M} \right\}
\end{multline}
where the particular form of $\iota_C$ immediately leads to $x_1^k=\dots=x_M^k$, and therefore, for all $i\in 1,\dots,M$,
\begin{align}
    x_i^k
     & = \argmin_{x\in\Constraints} \left\{ g(x) +  \frac{1}{M}\sum_{i=1}^M \left( f_i(x^{k-1}_i) + \langle x - x_i^{k-1} ; \nabla f_i(x_i^{k-1}) \rangle + \frac{1}{\gamma} D_h(x,x_i^{k-1}) \right)  \right\} \\
     & = \argmin_{x\in\Constraints} \left\{ \gamma g(x) + h(x) +    \left\langle x  ; \frac{1}{M}\sum_{i=1}^M \left( \gamma \nabla f_i(x_i^{k-1}) - \nabla h(x^{k-1})   \right) \right\rangle   \right\}.
\end{align}

We obtain an asynchronous variant of this iteration by following the same steps as in \cref{sec:smooth}.  First we replace $ \nabla f_i(x^{k-1})$
by $ \nabla f_i(x^{k-D_i^k})$ for each worker $i$. Second, we have two choices for replacing $\nabla h(x^{k-1})$:
\begin{itemize}
    \item either it is kept at the last iterate $x^{k-1}$, which leads to the Bregman version of the PIAG method, which appears in \cite[Sec. V]{aytekin2016analysis}: \begin{align}
              x^k & = \argmin_{x\in\Constraints} \left\{ \gamma g(x) + h(x) +    \left\langle x  ; \frac{1}{M}\sum_{i=1}^M \left( \gamma \nabla f_i(x_i^{k-D_i^k}) - \nabla h(x^{k-1})   \right) \right\rangle   \right\} \label{alg:piagBregman}
          \end{align}
    \item or, as prescribed here, we set it to the point corresponding to the gradient computation $x^{k-D_i^k}$, which leads to:
          \begin{align}
              x^k & = \argmin_{x\in\Constraints} \left\{ \gamma g(x) + h(x) +    \left\langle x  ; \frac{1}{M}\sum_{i=1}^M \left( \gamma \nabla f_i(x_i^{k-D_i^k}) - \nabla h(x^{k-D_i^k})   \right) \right\rangle   \right\}  .
          \end{align}
\end{itemize}
We consider here the second option, inspired by the better results of \cite{mishchenko2020distributed} for the Euclidean case. Note also that the theory developed in \cite{aytekin2016analysis} for the first option is based on the strong assumption that the divergence $D_h$ should be lower and upper bounded by the squared Euclidean distance, which is not the case in \Cref{example} for example.

\subsection{Algorithm \& Practical implementation}

We are now ready to state our asynchronous Bregman proximal-gradient algorithm. First, the central machine initializes $x^0\in\interior\dom h $ and sends it to all workers. Then, the central machine keeps track of the quantity
\begin{align}
    \overline{u}^k :=  \frac{1}{M} \sum_{i=1}^M \underbrace{ \gamma \nabla  f_i(x^{k-D_i^k}) - \nabla h(x^{k-D_i^k}) }_{:=u^k_i}
\end{align}
where  $u_i^k$ corresponds to the last contribution of agent $i$, received at time $k-d_i^k$ (hence $u_i^k=u_i^{k-1}=\cdots=u_i^{k-d_i^k}$), which results from its local computation from point $x^{k-D_i^k}$.

A key feature of our algorithm emerges here: regardless of each worker's response delay $D_i^k$, their contribution to the central node's aggregate $\overline{u}^k$ remains constant. As we will illustrate in our subsequent analysis, this allows us to select a constant step-size $\gamma$ that is independent of any hypothetical bound on these delays.

Along the algorithm, an iteration is triggered as soon as the central machine receives a communication. At the $k$-th iteration, with an incoming call from worker $i$, the central machine:
\begin{itemize}
    \item receives a contribution $u_i$ from agent $i$,
    \item updates $\overline{u}^k$ by setting $u_i^k = u_i $ and $u_{i'}^k = u_{i'}^{k-1} $ for all $i'\neq i $,
    \item computes the new point
          \begin{align}
              x^{k} = \argmin_{x\in\Constraints} \left\{   h(x) + \gamma g(x) +  \left\langle \overline{u}^k , x \right\rangle  \right\}
          \end{align}
    \item sends $x^k$ to agent $i$.
\end{itemize}

This algorithm is described in \Cref{alg:dave_md}, where the workers communicate only the adjustments between two iterates (to avoid storing all $u_i$ at the central machine).

\tcbset{width=1.0\columnwidth,before=,after=, colframe=black,colback=white, fonttitle=\bfseries, coltitle=white, colbacktitle=red!80!yellow, boxrule=0.2mm}
\begin{algorithm}[H]
    \caption{\texttt{Asynchronous Bregman proximal-gradient algorithm}\label{alg:dave_md}}
    \centering
    \begin{multicols}{2}
        \begin{tcolorbox}[title=Central machine:]
            Initialize $\overline u$, $k=0$\\
            {\color{red!80!yellow} Send $\overline u$ to all workers}\\
            \While{not stopped}{
            \textbf{when} a worker finishes:\\
            {\color{blue!70!black}Receive adjustment $\Delta$ from it}\\
            $\overline{u} \leftarrow \overline{u} + \frac{\Delta}{M}$\\
            \resizebox{0.96\textwidth}{!}{$\displaystyle x \leftarrow \argmin_{x\in\Constraints } h(x) + \gamma g(x) +  \left\langle\overline{u},x\right\rangle  $}\\
            {\color{red!80!yellow} Send $x$ to the agent in return}\\
            $k\leftarrow k+1$
            }
            Interrupt all slaves\\
            \textbf{Output} $ x$
        \end{tcolorbox}

        \columnbreak
        \tcbset{width=1.03\columnwidth,before=\hspace{-0.2cm}, colframe=black!50!black, colbacktitle=blue!70!black}
        \begin{tcolorbox}[title=Worker $i$:]
            Initialize $u = \overline u$\\
            \While{not interrupted by central machine}{
            {\color{red!80!yellow}Receive the most recent $x$}\\
            $u^+ \leftarrow \gamma  \nabla f_i(x) - \nabla h(x)$\\
            $\Delta \leftarrow u^+ - u $\\
            $u \leftarrow u^+$\\
            {\color{blue!70!black}Send adjustment $\Delta$ to central machine}
            }
        \end{tcolorbox}
    \end{multicols}
\end{algorithm}

\section{Analysis of the algorithm}\label{sec:analysis}

The analysis of the algorithm described in the last section consists of two parts, to which the next two subsections are devoted: we establish first that the iterates are well-defined, and second that they converge to the minimum of \eqref{eq:base}, under no specific assumption on the computing system.

\subsection{The algorithm is well-defined}\label{sec:well}

To show that the iterates produced by \Cref{alg:dave_md} are valid, we first show a generic result on the Bregman proximal gradient operator, which is at the core of our developments.  The proof of this result simply consists in applying, in the product space $\bigtimes_{i=1}^M \Constraints$, the general well-posedness result of Bregman proximal gradient operator of \cite[Lemma 2]{bauschke2017descent}.

\begin{lemma}[Well-posedness]\label{lem:defined}
    Let \cref{asm:pb,asm:reg,asm:fun} hold. For any stepsize $\gamma >0$ and any vector $y=(y_1,\dots,y_M)\in  \bigtimes_{i=1}^M \interior\dom h$, the operator
    \begin{align}
        T_\gamma(y)  := \argmin_{x\in\Constraints} \left\{ \gamma g(x) + h(x) +    \left\langle x  ; \frac{1}{M}\sum_{i=1}^M \left( \gamma \nabla f_i(y_i) - \nabla h(y_i)   \right) \right\rangle   \right\}
    \end{align}
    is non-empty and single-valued in $\interior \dom h$.
\end{lemma}

\begin{proof}
    Let $\gamma >0$ and $y=(y_1,\dots,y_M)\in  \bigtimes_{i=1}^M \interior\dom h$, let us define the operator

    \begin{multline}
        \mathbb{T}_\gamma(y) := \argmin_{x\in\bigtimes_{i=1}^M \Constraints}
        \left\{ \frac{1}{M} \sum_{i=1}^M g(x_i) + \frac{1}{M} \sum_{i=1}^M \left[ f_i(y_i) + \left\langle \nabla f_i(y_i) , x_i - y_i \right\rangle \right] + \dots \right. \\
        \left. \dots  + \frac{1}{\gamma M}  \sum_{i=1}^M D_h(x_i,y_i) + \iota_C(x)
        \vphantom{\sum_{i=1}^M} \right\} \label{eq:TT}
    \end{multline}
    with $\iota_C$ the consensus indicator defined by \eqref{eq:C}.
    By applying the same reasoning as in \cref{sec:bregmanasync}, it is easy to see that $\mathbb{T}_\gamma(y) = (T_\gamma(y),\dots,T_\gamma(y))$.

    Under \cref{asm:pb,asm:reg,asm:fun}, we have that i) $x\mapsto \iota_C(x) + \frac{1}{M} \sum_{i=1}^M g(x_i)$ is proper, lsc., and convex on $\bigtimes_{i=1}^M \Constraints$; ii)  $x\mapsto \frac{1}{M} \sum_{i=1}^M h(x_i)$ is Legendre on $\bigtimes_{i=1}^M \Constraints$; iii)  $x\mapsto \frac{1}{M} \sum_{i=1}^M f_i(x_i)$ is proper, lsc., convex, and differentiable on $\interior\dom h$; iv) $\dom \iota_C \cap (\dom g)^M \cap \bigtimes_{i=1}^M \interior\dom h  \neq \emptyset$; and v) the problem $\inf_{x\in\ProdSpace}\{ \iota_C(x) + \frac{1}{M} \sum_{i=1}^M (f(x_i) + g(x_i)) \}$ has a non-empty compact solution set. These conditions enable us to apply \cite[Lem. 2]{bauschke2017descent} which gives that $\mathbb{T}_\gamma(y)$ is non-empty, single-valued, and maps $ \bigtimes_{i=1}^M \interior\dom h  $ to $ \bigtimes_{i=1}^M \interior\dom h$.
    In turn, we obtain that $T_\gamma(y)$ is non-empty and single-valued in $\interior \dom h$. \end{proof}

Thus, noticing that in our algorithm, the central machine generates a new iterate $x^k$, as the point produced by applying $T_\gamma$ to $ y^k=(x^{k-D_1^k},\dots,x^{k-D_M^k}) $, ie.
\begin{align}
    x^{k} = T_\gamma(y^k) = \argmin_{x\in\Constraints} \left\{ \gamma g(x) + h(x) +    \left\langle x  ; \overbrace{\frac{1}{M}\sum_{i=1}\underbrace{\left( \gamma \nabla f_i(x^{k-D_i^k}) - \nabla h(x^{k-D_i^k})   \right)}_{:= u_i^k}}^{\overline{u}^k} \right\rangle   \right\}.
\end{align}
Thus \cref{lem:defined} implies that our algorithm is well-defined, from an initial iterate $x^0\in \interior \dom h$. Indeed, the points $(x^{k'})_{k'<k}$ were generated by the central machine and thus belongs to $\interior \dom h $; in turn, the input at iteration $k$ is  $y^k=(x^{k-D_1^k},\dots,x^{k-D_M^k})$ and belongs to $ \bigtimes_{i=1}^M \interior \dom h $. The algorithm is thus well defined and produces points in $\interior\dom h$.

\subsection{Convergence result}

We now focus on the convergence of the algorithm, established in forthcoming \cref{th:conv}. The proof of convergence consists in carefully combining the contraction results of Bregman operators \cite{chen1993convergence,bauschke2017descent} together with the techniques developed in \cite{mishchenko2020distributed} for the asynchronous proximal algorithms.

\begin{lemma}[Contraction]\label{lem:descentasync}
    Let \cref{asm:pb,asm:reg,asm:fun} hold. For any stepsize $\gamma >0$ and any vector $y=(y_1,\dots,y_M)\in  \bigtimes_{i=1}^M \interior\dom h$, and any $u\in\Constraints$,
    \begin{align}
        D_{{{h}}}(u,T_\gamma(y)) & \leq \frac{1}{M} \sum_{i=1}^M D_h(u,y_i)  - \frac{1-\gamma L}{M} \sum_{i=1}^M D_h(T_\gamma(y),y_i) - \gamma ( F(T_\gamma(y)) -  F(u) ) .
    \end{align}
\end{lemma}
\begin{proof}
    The proof follows the same reasoning as the one of \Cref{lem:defined}: working on the product space and deriving properties on the operator $\mathbb{T}_\gamma$. By applying \cite[Lem. 5]{bauschke2017descent} (whose assumptions are verified as in \Cref{lem:defined}) with $(u,\dots,u)\in\Constraints^M$ and $\mathbb{T}_\gamma(y)$,

    \begin{multline}
        \frac{1}{M} \sum_{i=1}^M D_h(u,[\mathbb{T}_\gamma(y)]_i) \leq \frac{1}{M} \sum_{i=1}^M D_h(u,x_i) -  \frac{1-\gamma L}{M} \sum_{i=1}^M D_h([\mathbb{T}_\gamma(y)]_i,y_i) \dots \\
        ~~ \dots - \gamma \left( \iota_C(\mathbb{T}_\gamma(y)) + \frac{1}{M}\sum_{i=1}^M f_i([\mathbb{T}_\gamma(y)]_i) + g([\mathbb{T}_\gamma(y)]_i) \dots \right. \\
        \left. \dots- \iota_C(u,\dots,u) - \frac{1}{M}\sum_{i=1}^M f_i(u) + g(u) \right)
    \end{multline}

    and then, since  $\mathbb{T}_\gamma(y) = (T_\gamma(y),\dots,T_\gamma(y))$, we obtain the claimed result.
\end{proof}

Since any solution $x^\star$ of \eqref{eq:pb} belongs to $\dom h = \Constraints$, applying \Cref{lem:descentasync} to the iterates produced by our algorithm with $u=x^\star$ shows that for all $k$,
\begin{align}
    \label{eq:descent_algo}
    D_{{{h}}}(x^\star,x^{k}) & \leq \frac{1}{M} \sum_{i=1}^M D_h(x^\star,x^{k-D_i^k})  - \frac{1-\gamma L}{M} \sum_{i=1}^M D_h(x^{k},x^{k-D_i^k})\dots \\
                             & ~~~~~~~~~ \cdots- \gamma \left( F(x^k) - F(x^\star) \right)
\end{align}
which is central in the proof of the following convergence result.
The pointwise convergence arguments in the proof of this result are based on Opial-type arguments for which the following additional condition is needed. This assumption is verified for many usual divergences, including the Boltzmann-Shannon entropy, the Hellinger distance, etc.
It corresponds to assumption \textbf{H} in \cite{bauschke2017descent}.

\begin{asm}[Level boundedness and limit]\label{asm:addh}
    For any $x\in\dom h$ and $t\in\mathbb{R}$, the set $\{y\in\interior\dom h : D_h(x,y)\leq t\}$ is bounded. In addition, $x^k\to x$ if and only if $D_h(x,x^k)\to 0$.
\end{asm}

\newpage

A key feature of the algorithm is that its convergence does not rely on prior knowledge or assumptions on the answering delays. It is only expected that the workers are never dropped, meaning that they all have a finite answering time. Let us introduce this assumption formally.

To do so we need to define the concept of ``epoch" denoted $m \in \mathbb{N}$, and of the first iteration $k_m$ of an ``epoch" $m$ (drawing from \cite{icml2018}).
\begin{itemize}
    \item An iteration $k$ belongs to epoch $m$ if it falls in the range $k_{m} \le k < k_{m+1}$
    \item The epoch $m$ starts at $0$ (together with $k$) and is incremented whenever the central node has made at least one \emph{full interaction} with each worker during the current epoch. Namely, whenever all workers have answered at least one query sent after $k_m$. Since at iteration $k$, the most recent computation by a worker $i$ was requested at iteration $k - D_i^k$, this condition therefore translates to $k_m \le k - D_i^k$:
          \begin{align}
              k_{m} & =\min\{k:\text{at least 1 full interaction per worker since }k_{m-1}\} \\
                    & = \min\{k: k-D_i^k \geq k_{m-1} \text{ for all } i=1,..,M \}.
          \end{align}
\end{itemize}

\begin{asm}[Worker participation]\label{asm:answer}
    Workers have a finite answering time, in other words, $m \underset{k \rightarrow +\infty}{\longrightarrow} +\infty$.
\end{asm}
Note that when \cref{asm:answer} does not hold, we have “a subset of workers ceases to respond” and this leads to a different problem: in such cases, the non-responsive workers can be discarded, and the analysis is restricted to the active workers.

We are now ready to formulate the main result of this paper, namely, the convergence of our asynchronous algorithm.

\begin{theorem}[Convergence]\label{th:conv}
    Let \cref{asm:pb,asm:reg,asm:fun,asm:addh,asm:answer}, hold and fix $\gamma \in (0,1/L)$. If \eqref{eq:pb} has a unique minimizer $x^\star$, then the sequence $(x^{k})$ generated by
    \Cref{alg:dave_md} converges to $x^\star$.
\end{theorem}

\begin{proof}
    Let us start with using~\eqref{eq:descent_algo}. Since $\gamma \in (0,1/L)$, we can disregard the second term involving $D_h(x^{k},x^{k-D_i^k})$ which is negative. Then we get that for any $k\geq k_{m}$ the following bound:
    \begin{align}
        D_{{{h}}}(x^\star,x^{k}) + \gamma \left(F(x^k) - F(x^\star) \right)
        \leq \max_{i}   D_h(x^\star,x^{k-D_i^k}) \leq \max_{k'\in[k_{m-1},k)}\!D_h(x^\star,x^{k'}).\label{eq:bound}
    \end{align}
    In particular, for $k=k_m$, this implies
    \begin{align}
        D_{{{h}}}(x^\star,x^{k_m}) & \leq  \max_{k'\in[k_{m-1},k_m)}   D_h(x^\star,x^{k'}). \label{eq:base}
    \end{align}

    Applying now \eqref{eq:bound} for $k=k_m+1$, we obtain
    \begin{align}
        D_{{{h}}}(x^\star,x^{k_m+1}) + \gamma \left( F(x^{k_m+1}) - F(x^\star) \right) & \leq \max\left\{  D_{{{h}}}(x^\star,x^{k_m}) ,   \max_{k'\in[k_{m-1},k_m)}   D_h(x^\star,x^{k'}) \right\} \\
        & \leq    \max_{k'\in[k_{m-1},k_m)}   D_h(x^\star,x^{k'})
    \end{align}
    where we use \eqref{eq:base} for the second inequality. Repeating this recursively, we obtain for all $k\in [k_m, k_{m+1})$
    \begin{align}
        \big(0 \leq \,\big) ~D_{{{h}}}(x^\star,x^{k}) + \gamma \left( F(x^{k}) - F(x^\star) \right) & \leq  \max_{k'\in[k_{m-1},k_m)}   D_h(x^\star,x^{k'}). \label{eq:bound2}
    \end{align}
    Observe that this bound yields
    \begin{align}
        \label{eq:bk}
        \max_{k\in [k_m, k_{m+1})}  D_{h}(x^\star,x^{k}) & \leq
        \max_{k'\in[k_{m-1},k_m)}   D_h(x^\star,x^{k'}).
    \end{align}
    In words, the non-negative sequence $\max_{k\in [k_m, k_{m+1})}  D_{h}(x^\star,x^{k})$
    is non-increasing and thus converges to a non-negative value.

    Consider now a sequence of time indices $(l^m)$ where the above maximum is attained:
    $$ l^m \in \arg\max_{k\in [k_m, k_{m+1})}  D_{h}(x^\star,x^{k}).$$
    Since $l^m \in [k_m, k_{m+1})$, we have
    \begin{align}
        D_{{{h}}}(x^\star,x^{l^m}) + \gamma \left( F(x^{l^m}) - F(x^\star) \right) \leq    \max_{k\in [k_m, k_{m+1})}  D_{{{h}}}(x^\star,x^{k}) + \gamma \left( F(x^{k}) - F(x^\star) \right)
    \end{align}
    which gives, with the help of \eqref{eq:bound2},
    \begin{align}
        D_{{{h}}}(x^\star,x^{l^m}) + \gamma \left( F(x^{l^m}) - F(x^\star) \right) & \leq   \max_{k'\in[k_{m-1},k_m)}   D_h(x^\star,x^{k'}) ~=~ D_h(x^\star,x^{l^{m-1}}) .
    \end{align}
    By letting $m\to \infty$ in this inequality, we obtain
    \begin{align}
        \lim_{m\to\infty} F(x^{l^m}) = F(x^\star).\label{eq:lim}
    \end{align}
    Since we have $D_{{{h}}}(x^\star,x^{l^m})\leq D_{{{h}}}(x^\star,x^{0})$, \cref{asm:addh}
    yields that $(x^{l^m})$ is bounded, so that we can extract a sub-sequence that converges to a point $\tilde x\in\dom h$. From \eqref{eq:lim}, $\tilde x$ must be a minimizer for $F$, and by uniqueness, we have $\tilde x = x^\star$.  Thus we have $D_{h}(x^\star,x^{l^m}) \longrightarrow_{m\to\infty} 0$.
    This allows us to conclude from~\eqref{eq:bound2}: the right-hand side vanishes, so that $D_{{{h}}}(x^\star,x^{k}) \to 0$ and  $F(x^{k}) \to F(x^\star)$.

\end{proof}

\section{Numerical illustrations}\label{sec:num}

In this section, we take a look at the practical behavior of our asynchronous Bregman algorithm on a simple synthetic distributed problem. We provide an illustration of the convergence result and a  comparison with competing algorithms. A complete computational study is out of the scope of this paper.
The numerical illustrations are operated in \texttt{Julia} \cite{bezanson2017julia} on a personal computer. We have packaged a toolbox implementing the algorithms and the experiments; it is publicly available at \url{github.com/Selim78/distributed-bregman}.

\subsection{Experimental set-up}\label{sec:setup}

We illustrate our algorithm on a regularized non-negative linear regression problem, which is a distributed variant of the problem of \cite[Sec.\;5.3]{bauschke2017descent} (see also \cite{csiszar1991least}). It consists in the minimization of the function of \cref{example} with an additional $\ell_1$-regularization to ensure that the problem has a unique minimizer.
Let $A\in\mathbb{R}_+^{m\times n}$ be a sensing matrix and $b\in\mathbb{R}_{++}^m$ a positive output vector. Assume that we split the $m$ samples of this dataset on $M$ workers (for simplicity, we consider in our experiments an even partition of $m/M$ examples).
With $a_j$ representing the $j$-th row of $A$, the objective function writes
\begin{align}\label{eq:pbnum}
    \min_{x\in\mathbb{R}_+^n} ~\frac{1}{M} \sum_{i=1}^M \underbrace{\left( \sum_{j=\frac{(i-1)m}{M}+1}^{\frac{im}{M}} \langle a_j , x \rangle \log \langle a_j , x \rangle - (\log b_j +1 )  \langle a_j , x \rangle + b_j  \right)}_{f_i(x)} + \lambda \|x\|_1.
\end{align}

In practice, the data is generated as follows. We take $n=100$ and $m=200$; the rows of matrix $A$ are drawn from a uniform distribution in $[0,1)$. The vector $b$ is generated as $b= A\bar x + \epsilon$ with a random positive $\bar x$ plus a Poisson noise with rate $1$.
This problem is distributed on $M=10$ processes living in separate memory domains using \texttt{Julia}'s Distributed module, which allows the control of process generation and communication between processes. To simulate some heterogeneity between the workers, we artificially slow down two workers by a factor\;$5$\;and\;$10$ respectively.

We solve this problem with the  three following Bregman algorithms using the Boltzmann-Shannon entropy $h(x) = \sum_{j=1}^n x_j \log x_j$ (since the first part of the objective function is $L$-smooth with respect to it with $L$ as in \eqref{eq:L} as discussed in \cref{example}):
\begin{itemize}
    \item the synchronous algorithm based on the iteration \eqref{eq:gradprox} (called Synchronous),
    \item the asynchronous variant based on the iteration \eqref{alg:piagBregman} (called Bregman-PIAG),
    \item our asynchronous variant (\Cref{alg:dave_md}).
\end{itemize}
We note that our asynchronous algorithm, as well as the synchronous one, is guaranteed to converge with the standard stepsize $\gamma < 1/L$. In contrast, Bregman-PIAG has no convergence guarantee, as recalled at the end of \cref{sec:bregmanasync}. In practice, we run all algorithms with, both, a theory-complying $0.99/L$ stepsize and a tuned stepsize.

\begin{figure}[!ht]
    \centering
    \resizebox{0.49\textwidth}{!}{\includegraphics{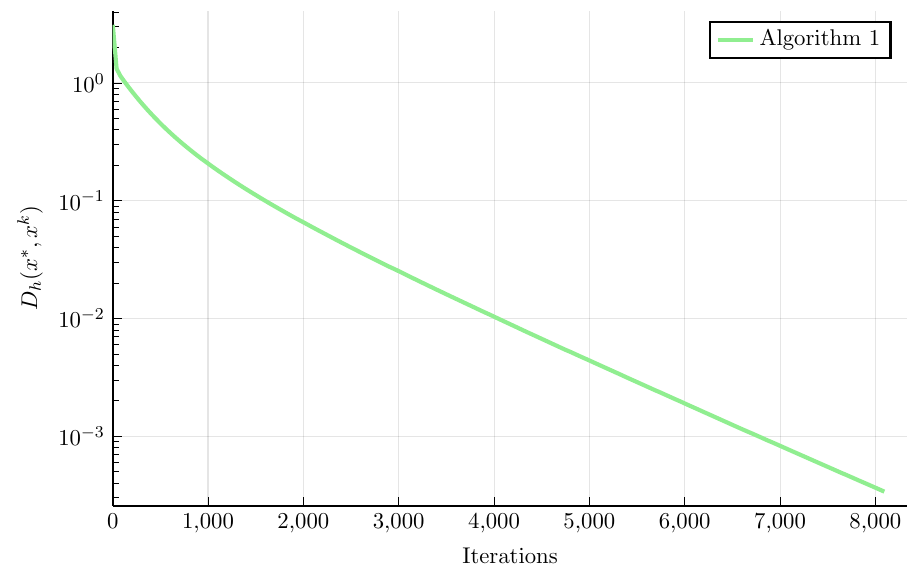}}
    \resizebox{0.49\textwidth}{!}{\includegraphics{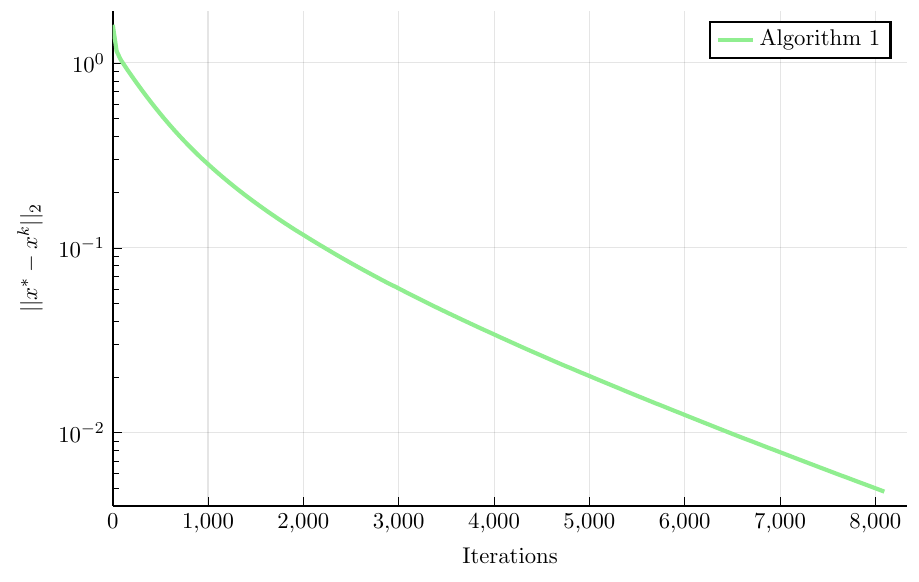}}
    \vspace*{-2ex}
    \caption{Convergence with respect to the number of iterations, illustrating \Cref{th:conv}}
    \label{fig:1}
\end{figure}

\vspace*{-2ex}

\subsection{Experimental results}\label{sec:res}

In \Cref{fig:1}, we show convergence of the iterates of \Cref{alg:dave_md} (with $\gamma=0.99/L$), as guaranteed by \Cref{th:conv}. The two plots display a decrease of, respectively, $D_h(x^k,x^\star)$ and $\| x^k-x^\star \|^2$ over iterations\footnote{The optimal solution $x^\star$ is obtained by running the synchronous algorithm for a long time.}. We see on the plots that this decrease is mostly monotonous; we note that the proof of \Cref{th:conv} establishes monotonicity for the worst Bregman divergence over epochs \eqref{eq:bk}.

In \Cref{fig:2}, we report the performance for the three algorithms described in the previous section. Since iterations have two different meanings for asynchronous or synchronous algorithms, we compare the convergence speed in terms of wallclock time. On the left-hand plot, we use the theoretical stepsize $\gamma$ for the three algorithms, and on the right-hand plot, we use tuned stepsizes\footnote{For each algorithm, $\gamma$ is chosen among the exponential of $1, 1.5, …, 2.5, 3$. The $\gamma$ retained is the one which produces the last iterate with the minimum Bergman distance to $x^*$ after a $100$ second run.}.
We notice a clear gain in performance with our approach, in the two cases. Note finally that the instance of the problem, generated as described in \Cref{sec:setup}, is only weakly heterogeneous. We choose it on purpose to compare our algorithm with two other algorithms in a situation that is not, apriori, the best for our algorithm. Even better performances can be obtained for stronger heterogeneity of the data or the system. A complete computational study is out of the scope of this paper, which is mainly methodological.

\begin{figure}[!h]
    \centering
    \resizebox{0.49\textwidth}{!}{\includegraphics{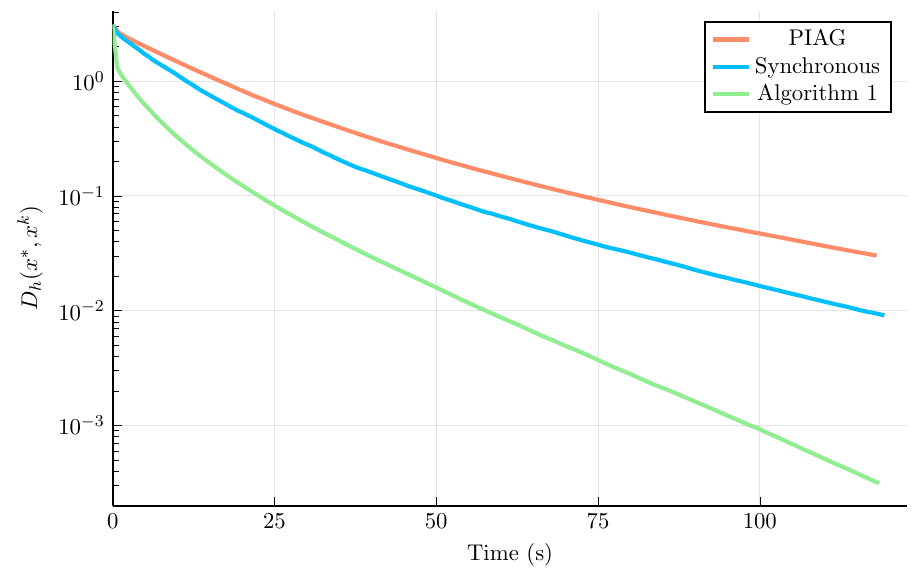}}
    \resizebox{0.49\textwidth}{!}{\includegraphics{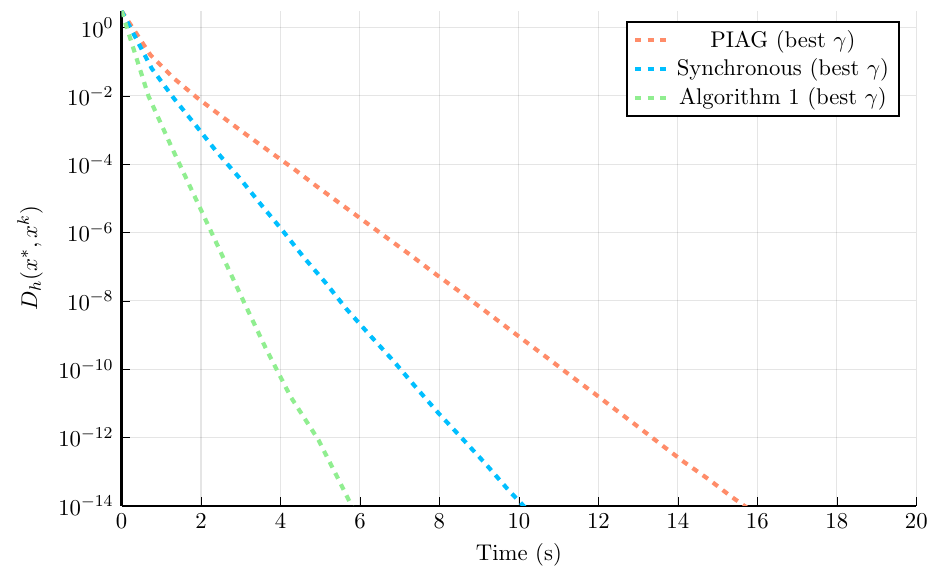}}
    \vspace*{-2ex}
    \caption{Comparison of our algorithm with two existing ones, with respect to wall-clock time}
    \label{fig:2}
\end{figure}

\section{Conclusions, perspectives}

In this paper, we provided an asynchronous version of the Bregman proximal gradient method. Building on efficient asynchronous methods for the Euclidean case and on smoothness models adapted to Bregman geometries, we derive and analyze a method that can handle any kind of delays, with a simple implementation using the same step size as in the synchronous case.
The light assumption on the delays combined with the subtlety of Bregman geometry make the analysis rather involved: we were able to show the convergence of our method whenever there is a unique minimizer to our problem; nevertheless, we believe that finer results could be derived using the same proof template. For instance, it may be possible to obtain a local convergence rate in the strongly convex case using the Legendre exponent reasoning of \cite{azizian2021last}.

\section*{Acknowledgements}

This work has been supported by MIAI Grenoble Alpes (ANR-19-P3IA-0003).

\bibliographystyle{spmpsci}  
\bibliography{references}

\begin{thebibliography}{10}
\providecommand{\url}[1]{{#1}}
\providecommand{\urlprefix}{URL }
\expandafter\ifx\csname urlstyle\endcsname\relax
  \providecommand{\doi}[1]{DOI~\discretionary{}{}{}#1}\else
  \providecommand{\doi}{DOI~\discretionary{}{}{}\begingroup
  \urlstyle{rm}\Url}\fi

\bibitem{auslender2006interior}
Auslender, A., Teboulle, M.: Interior gradient and proximal methods for convex
  and conic optimization.
\newblock SIAM Journal on Optimization \textbf{16}(3), 697--725 (2006)

\bibitem{aytekin2016analysis}
Aytekin, A., Feyzmahdavian, H.R., Johansson, M.: Analysis and implementation of
  an asynchronous optimization algorithm for the parameter server.
\newblock arXiv:1610.05507  (2016)

\bibitem{azizian2021last}
Azizian, W., Iutzeler, F., Malick, J., Mertikopoulos, P.: The last-iterate
  convergence rate of optimistic mirror descent in stochastic variational
  inequalities.
\newblock In: Conference on Learning Theory, pp. 326--358. PMLR (2021)

\bibitem{bauschke2017descent}
Bauschke, H.H., Bolte, J., Teboulle, M.: A descent lemma beyond lipschitz
  gradient continuity: first-order methods revisited and applications.
\newblock Mathematics of Operations Research \textbf{42}(2), 330--348 (2017)

\bibitem{bauschke1997legendre}
Bauschke, H.H., Borwein, J.M., et~al.: Legendre functions and the method of
  random bregman projections.
\newblock Journal of convex analysis \textbf{4}(1), 27--67 (1997)

\bibitem{bauschke2011convex}
Bauschke, H.H., Combettes, P.L.: Convex analysis and monotone operator theory
  in Hilbert spaces.
\newblock Springer Science \& Business Media (2011)

\bibitem{beck2017first}
Beck, A.: First-order methods in optimization.
\newblock SIAM (2017)

\bibitem{bertero2009image}
Bertero, M., Boccacci, P., Desider{\`a}, G., Vicidomini, G.: Image deblurring
  with poisson data: from cells to galaxies.
\newblock Inverse Problems \textbf{25}(12), 123006 (2009)

\bibitem{bertsekas1989parallel}
Bertsekas, D.P., Tsitsiklis, J.N.: Parallel and distributed computation:
  numerical methods, vol.~23.
\newblock Prentice hall Englewood Cliffs, NJ (1989)

\bibitem{bezanson2017julia}
Bezanson, J., Edelman, A., Karpinski, S., Shah, V.B.: Julia: A fresh approach
  to numerical computing.
\newblock SIAM review \textbf{59}(1), 65--98 (2017)

\bibitem{bolte2018first}
Bolte, J., Sabach, S., Teboulle, M., Vaisbourd, Y.: First order methods beyond
  convexity and lipschitz gradient continuity with applications to quadratic
  inverse problems.
\newblock SIAM Journal on Optimization \textbf{28}(3), 2131--2151 (2018)

\bibitem{bregman1967relaxation}
Bregman, L.M.: The relaxation method of finding the common point of convex sets
  and its application to the solution of problems in convex programming.
\newblock USSR computational mathematics and mathematical physics
  \textbf{7}(3), 200--217 (1967)

\bibitem{chen1993convergence}
Chen, G., Teboulle, M.: Convergence analysis of a proximal-like minimization
  algorithm using bregman functions.
\newblock SIAM Journal on Optimization \textbf{3}(3), 538--543 (1993)

\bibitem{csiszar1991least}
Csiszar, I.: Why least squares and maximum entropy? an axiomatic approach to
  inference for linear inverse problems.
\newblock The annals of statistics \textbf{19}(4), 2032--2066 (1991)

\bibitem{hannah2017more}
Hannah, R., Yin, W.: More iterations per second, same quality--why asynchronous
  algorithms may drastically outperform traditional ones.
\newblock arXiv:1708.05136  (2017)

\bibitem{hannah2016unbounded}
Hannah, R., Yin, W.: On unbounded delays in asynchronous parallel fixed-point
  algorithms.
\newblock Journal of Scientific Computing \textbf{76}(1), 299--326 (2018)

\bibitem{lu2017relativelysmooth}
Lu, H., Freund, R.M., Nesterov, Y.: Relatively-smooth convex optimization by
  first-order methods, and applications (2017)

\bibitem{mishchenko2020distributed}
Mishchenko, K., Iutzeler, F., Malick, J.: A distributed flexible delay-tolerant
  proximal gradient algorithm.
\newblock SIAM Journal on Optimization \textbf{30}(1), 933--959 (2020)

\bibitem{icml2018}
Mishchenko, K., Iutzeler, F., Malick, J., Amini, M.R.: A delay-tolerant
  proximal-gradient algorithm for distributed learning.
\newblock In: Proceedings of the 35th international conference on machine
  learning (ICML) (2018)

\bibitem{nesterov-book}
Nesterov, Y.: Introductory lectures on convex optimization: A basic course,
  vol.~87.
\newblock Springer Science \& Business Media (2013)

\bibitem{rockafellar2015convex}
Rockafellar, R.T.: Convex analysis.
\newblock Princeton university press (2015)

\bibitem{sun2017asynchronous}
Sun, T., Hannah, R., Yin, W.: Asynchronous coordinate descent under more
  realistic assumptions.
\newblock In: Advances in Neural Information Processing Systems, pp. 6182--6190
  (2017)

\bibitem{vanli2018global}
Vanli, N.D., G\"{u}rb\"{u}zbalaban, M., Ozdaglar, A.: Global convergence rate
  of proximal incremental aggregated gradient methods.
\newblock SIAM Journal on Optimization \textbf{28}(2), 1282--1300 (2018)

\end{thebibliography}

\end{document}